\documentclass[12pt]{amsart} 

\usepackage{amsmath, amssymb}
\usepackage{amsthm}

\theoremstyle{plain}
\newtheorem{theorem}{\indent\sc Theorem}[section]
\newtheorem{lemma}[theorem]{\indent\sc Lemma}
\newtheorem{corollary}[theorem]{\indent\sc Corollary}

\newtheorem{claim}[theorem]{\indent\sc Claim}
\newtheorem{conjecture}[theorem]{\indent\sc Conjecture}
\theoremstyle{definition}
\newtheorem{definition}[theorem]{\indent\sc Definition}
\newtheorem{remark}[theorem]{\indent\sc Remark}

\newcommand{\vol}{\operatorname{vol}}
\newcommand{\Det}{\operatorname{Det}}

\newcommand{\lct}{\operatorname{lct}}

\renewcommand{\theenumi}{$\mathrm{\roman{enumi}}$}

\title[K-stability of Fano varieties and anticanonical divisors]
{On the K-stability of Fano varieties and anticanonical divisors} 
\author{Kento Fujita, Yuji Odaka}
\date{}

\begin{document}

\maketitle

%%%%%%%%%%%%%%% footnote %%%%%%%%%%%%%%%%
\footnote{ 2010 \textit{Mathematics Subject Classification}.
Primary 14J45; Secondary 14L24.
}
\footnote{ \textit{Key words and phrases}.
Fano varieties, K-stability, K\"ahler-Einstein metrics
}
\footnote{K.F.\ is Partially supported by JSPS Fellowship for Young Scientists.
Y.O.\ is Partially supported by JSPS Kakenhi no. 30700356 (Wakate (B)).
}
%%%%%%%%%%%%%%%%%%%%%%%%%%%%%%%%%%%%%%%%%

\begin{abstract}
We apply a recent theorem of Li and the first author to 
give some criteria for the K-stability of Fano varieties in terms of anticanonical $\mathbb{Q}$-divisors. 
First, we propose a condition in terms of certain anti-canonical $\mathbb{Q}$-divisors 
of given Fano variety, which we conjecture to be equivalent to the K-stability. 
We prove that it is at least sufficient condition 
and also relate to the Berman-Gibbs stability. 
We also give another algebraic proof of the K-stability of Fano varieties 
which satisfy Tian's alpha invariants condition. 
\end{abstract}

\section*{Introduction}

In this short paper, we discuss the K-stability of Fano varieties, which is an 
algebro-geometric stability condition originally motivated by studies of K\"ahler metrics. 
Indeed, as expected, when the base field is the complex number field, it is recently established that the existence of positive scalar curvature \textit{K\"ahler-Einstein metrics}, i.e., K\"ahler metrics with constant Ricci 
curvature, is actually equivalent to the algebro-geometric condition ``K-stability", 
by the works of \cite{DT92, tian1, don05, CT, stoppa, mab1, mab2, Ber12} and recent 
celebrated \cite{CDS1, CDS2, CDS3, tian2}. This equivalence had been known before 
as the Yau-Tian-Donaldson conjecture (for the case of Fano varieties). 

It also turned out that such canonical K\"ahler metrics and  
K-stability play crucial roles for nice moduli theory (cf., \cite{FS90,Od10,DS14,Od13b}), 
and indeed recently 
\cite{OSS16} constructed compact moduli spaces 
 of smoothable 
 K\"ahler-Einstein Fano varieties of two dimension and \cite{LWX14,SSY14,Od15} extended 
 to higher dimensional case. The current approaches heavily depend on \cite{DS14} and 
again the above-mentioned Yau-Tian-Donaldson equivalence. 

However, it has been known as a difficult problem to test K-stability for given Fano varieties. 
Our purpose here is to develop algebraic studies of K-stability of Fano varieties one step 
further, mainly after \cite{Li16,Fjt16}.

In this paper, we treat a stronger version of the K-stability introduced 
by Dervan \cite{Der14} 
and Boucksom-Hisamoto-Jonsson \cite{BHJ15} which is called 
uniform K-stability. The notion is expected to be eventually 
equivalent to the (original) K-stability. 

We start with fixing our notation as follows.

\subsection*{Notation}

\begin{itemize}

\item We work over an arbitrary algebraic closed field of characteristic zero throughout the paper. 

\item $X$ is a $\mathbb{Q}$-Fano variety of dimension $n$, 
which means a log terminal projective variety with ample $\mathbb{Q}$-Cartier anticanonical divisor. 

\item (cf., \cite[Definition 2.24]{KM98}) 
$F$ is a prime divisor over $X$, which means the equivalence class of an 
irreducible reduced Weil divisor on a normal blow up $Y$ of $X$ up to strict transform. 
In this paper, we occasionally denote the blow up by $\sigma\colon Y\to X$ . 
See \cite[\S 2.3]{KM98} for the details of the basic related materials. 

\item For $k\in\mathbb{Z}_{\geq 0}$ and $x\in\mathbb{R}_{\geq 0}$, let $H^0(-kK_X-xF)$ be the 
subspace of $H^0(-kK_X)$ whose sections vanishing along the generic point of $F$ 
at least $x$ times. 

\item For $x\in\mathbb{R}_{\geq 0}$, we set 
\[
\vol(-K_X-xF):=\limsup_{k\to\infty}\frac{h^0(-kK_X-kxF)}{k^n/n!}.
\]

\item $A_{X}(F)$ denotes the log discrepancy of $X$ along $F$ and 
$\tau(F)$ denotes the pseudo-effective threshold of $-K_{X}$ with respect to $F$, i.e., 
$$\tau(F):=\sup \left\{\frac{a}{k} \mid a,k\in \mathbb{Z}_{>0} \text{ s.t., } H^0(-kK_X-aF)\neq 0\right\}.$$
We also denote the log discrepancy of $(X, \Delta)$
along $F$ by $A_{(X, \Delta)}(F)$, where $\Delta$ is an $\mathbb{R}$-divisor on $X$ 
with $K_X+\Delta$ $\mathbb{R}$-Cartier.  

\item For an effective $\mathbb{R}$-Cartier divisor $D$ on $X$, 
\[
\lct(X; D):=\max\{c\in\mathbb{R}_{\geq 0}\,|\, (X, cD):\text{ log canonical}\}
\]
be the log canonical threshold of $X$ along $D$. 
\end{itemize}

A key notion we introduce is the following special type of anticanonical $\mathbb{Q}$-divisors. 

\begin{definition}\label{basis.div.def}
Let $k$ be a positive integer. 
Given any basis 
\[
s_{1},\dots,s_{h^{0}(-kK_{X})}\] 
of $H^{0}(-kK_{X})$, taking the corresponding 
divisors $D_{1},\dots,D_{h^{0}(-kK_{X})}$ ($D_{i}\sim -kK_{X}$), 
we get an anticanonical $\mathbb{Q}$-divisor 
\[
D:=\frac{D_{1}+\cdots+D_{h^{0}(-kK_{X})}}{k\cdot h^{0}(-kK_{X})}.
\] 
We call this kind of anticanonical $\mathbb{Q}$-divisor an 
\textit{anticanonical $(\mathbb{Q}$-$)$divisor of $k$-basis type}. 
\end{definition}

\begin{definition}\label{delta_dfn}
For $k\in\mathbb{Z}_{>0}$, set
\[
\delta_{k}(X):=\inf_{
\substack{
(-K_{X}\sim_\mathbb{Q}) D;\\
D:\,\,k\text{-basis type}
}} \lct(X;D).
\]
Moreover, we define 
\[
\delta(X):=\limsup_{k\to \infty}\delta_{k}(X).
\]
\end{definition}

Then we prove the following criterion.  
%(see Theorem \ref{paraph} for a slightly stronger statement). 

\begin{theorem}[K-stability criteria via basis type divisors]\label{paraph.intro}
Let $X$ be a $\mathbb{Q}$-Fano variety.  
If 
$\delta(X)> 1$ $($resp., $\ge 1)$
then $(X,-K_{X})$ is uniformly K-stable $($resp., K-semistable$)$.

\end{theorem}

Reviewing our proof of Theorem \ref{paraph}, we expect that the converse is also true though 
some technical difficulties is preventing us from proving it at the moment. 

\begin{conjecture}[Main conjecture]\label{Main.conj}
Let $X$ be a $\mathbb{Q}$-Fano variety. Then the K-stability $($resp., K-semistability$)$ 
of $(X,-K_{X})$ is equivalent to $\delta(X) > 1$ $($resp., $\delta(X) \ge 1)$.
\end{conjecture}

From Theorem \ref{paraph.intro}, we get that the Berman-Gibbs stability (resp., 
semistability) of $X$ implies K-stability (resp., K-semistability). 
The Berman-Gibbs stability was introduced in Berman \cite{Ber13} 
(see Definition \ref{Berman.Gibbs} for the definition) 
and the following algebraic result was known by \cite{fjtBG} (cf.\ \cite[\S 7]{Ber13}). 
However, we should emphasize that the proof of Theorem \ref{BGthm.intro} is 
much easier than the proof of \cite{fjtBG}. 

It also follows from our arguments 
that if the Berman-Gibbs stability is actually equivalent to the 
K-stability, it also implies 
our equivalence conjecture \ref{Main.conj}. 

\begin{theorem}[{cf.\ \cite{Ber13, fjtBG}}]\label{BGthm.intro}
Let $X$ be a $\mathbb{Q}$-Fano variety. If $X$ is Berman-Gibbs stable $($resp., 
Berman-Gibbs semistable$)$, then $X$ is uniformly K-stable $($resp., K-semistable$)$.
$($For the definition of Berman-Gibbs stability, see Section \ref{BG_section}.$)$ 
\end{theorem}

Finally, we also give a new algebraic (re-)proof of the following Tian's  famous criterion 
via the alpha invariant \cite{Tia87} (see Theorem \ref{alpha.K} in detail). 
The first algebraic proof of it is by the second author and Sano \cite{OS12} but our 
argument here is very different. Indeed, 
we work on $X$ itself 
and its valuations, thus in $n$-dimensional geometry, while the proof of \cite{OS12} was via 
analysis of test configurations, thus essentially depends on $(n+1)$-dimensional geometry. 

\begin{theorem}[cf., \cite{Tia87}, \cite{OS12}, \cite{Der14}, 
\cite{BHJ15}]\label{intro.alpha.K}
For an $n$-dimensional $\mathbb{Q}$-Fano variety $X$, 
if $\alpha(X)> (\text{resp., } \ge) n/(n+1)$, then $(X,-K_{X})$ 
is uniformly K-stable $($resp., K-semistable$)$. 
\end{theorem}

In this paper, 
we omit the original definition of the K-stability 
and for that, simply refer to the original \cite{tian1}, \cite{Don02}, and for uniform version, to \cite{Der14}, \cite{BHJ15}. 
The reason is that, as in the next section (Theorem \ref{LiFujita}) we start from review of the results of 
\cite{Fjt16}, \cite{Li16}  which can be seen as giving 
an alternative definition of the (uniform) K-(semi)stability of Fano varieties. 
Therefore, we do not need the original definition logically in this paper. 
After the review as Theorem \ref{LiFujita}, we also slightly modify the uniform K-stability part of it 
to the form we use in the following sections. 
In section $2$, we prove Theorem \ref{paraph.intro}, the criterion via 
basis type divisor (in the sense of \ref{basis.div.def}) 
and discuss relation with the Berman-Gibbs stability \cite{Ber13}. 
In the last section, we discuss the relation with the alpha invariant.

\subsection*{Acknowledgement}

We are grateful for 
Yuchen Liu who let the first author knows that he independently found 
Theorem \ref{alpha.K2}.

\section{Valuative criteria of K-stability and their variant} 

In this section, as a preparation, we recall the key theorem by \cite{Li16,Fjt16} and 
also prove a slight variant which we use in this paper. 

\begin{theorem}\label{LiFujita}
Let $X$ be a $n$-dimensional $\mathbb{Q}$-Fano variety. For an arbitrary prime divisor $F$ 
over $X$, 
we set 
\begin{eqnarray*}
\beta(F) &:= &A_{X}(F)(-K_{X})^{\cdot n}-\int_{x=0}^{\infty}\vol(-K_X-xF)dx,\\
j(F) &:= &\int_{x=0}^{\tau(F)}\left((-K_{X})^{\cdot n}-\vol(-K_{X}-xF)\right)dx.
\end{eqnarray*}
$($When $F$ is a bona fide divisor \emph{on} $X$, $\beta(F)$ first appeared in 
\cite{Fjt15}.$)$
Then we have 

\begin{enumerate}
\renewcommand{\theenumi}{{\rm \roman{enumi}}}
\renewcommand{\labelenumi}{$($\theenumi$)$}
\item\label{LiFujita1} 
$($\cite{Li16,Fjt16}$)$ $(X,-K_{X})$ is K-semistable if and only if 
$\beta(F)\ge 0$ for any $F$. Moreover, the pair is K-stable if 
$\beta(F)> 0$ for any $F$. 
\item\label{LiFujita2} 
$($\cite{Fjt16}$)$ $(X,-K_{X})$ is uniformly K-stable if and only if there exists a 
positive real number $\delta$ such that $\beta(F)\geq \delta\cdot j(F)$ for any $F$. 
\end{enumerate}
\end{theorem}

\noindent
The uniform K-stability treated in \eqref{LiFujita2} above is 
introduced by Dervan \cite{Der14}, Boucksom-Hisamoto-Jonsson 
\cite{BHJ15} as a conjecturally equivalent variant of the 
K-stability. \cite{Der14} refers to it as \textit{K-stability with respect to the minimum norm} 
and \cite{BHJ15} refers to it as \textit{J-uniform K-stability} or simply the uniform K-stability. 
In this paper, as there should be no confusion, we simply call it as the uniform K-stability. 
We also note that the above conditions for K-stability resembles 
the definition of log terminality, log canonicity. 

In this section, we prepare a variant of the above Theorem \ref{LiFujita} \eqref{LiFujita2}. 
Here, we prepare the following simple lemma.

\begin{lemma}\label{vol.tau}
$$(-K_{X})^{\cdot n}\tau(F)\ge \int_{x=0}^{\infty}\vol(-K_{X}-xF)dx 
\ge\frac{1}{n+1}(-K_{X})^{\cdot n}\tau(F).$$
\end{lemma}

\begin{proof}
The first inequality follows straightforward from $\vol(-K_{X}-xF)\le (-K_{X})^{\cdot n}$. 
The second inequality follows from the result of the concavity of volume function 
(cf., e.g., \cite{LM}); we can see the inequality
$\vol(-K_{X}-xF)\ge (-K_{X})^{\cdot n}(\frac{x}{\tau(F)})^{n}.$
\end{proof}

This Lemma \ref{vol.tau} allows us to give a version of the above 
Theorem \ref{LiFujita} \eqref{LiFujita2}, i.e., 
the uniform stability criterion in \cite{Fjt16}. 

\begin{theorem}\label{unif_thm}
Suppose that a $\mathbb{Q}$-Fano variety $X$ satisfies that there is a 
positive real constant $\varepsilon>0$ such that 
for any prime divisor $F$ over $X$, 
we have 
$$(1-\varepsilon)A_{X}(F)(-K_{X})^{\cdot n}\geq\int_{x=0}^{\infty}\vol(-K_{X}-xF)dx.$$

Then $(X,-K_{X})$ is uniformly K-stable. 
\end{theorem}

\begin{proof}
The assumption can be rewritten as, by Lemma \ref{vol.tau}, that there exists a positive real number 
$\varepsilon'$ satisfying 
\begin{equation}\label{ineq.1}
A_{X}(F)(-K_{X})^{\cdot n}\geq\int_{x=0}^{\infty}\vol(-K_{X}-xF)dx+\varepsilon'\tau(F)
(-K_X)^{\cdot n}, 
\end{equation}

\noindent
for any divisor $F$ over $X$. 
On the other hand, the desired inequality $\beta(F)> \delta\cdot j(F)$ of Theorem \ref{LiFujita} \eqref{LiFujita2} 
can be straightforwardly re-written as 

\begin{equation}\label{ineq.2}
\left((1+\delta')A_{X}(F)-\delta'\tau(F)\right)(-K_{X})^{\cdot n}\geq\int_{x=0}^{\infty}\vol(-K_{X}-xF)dx,
\end{equation}

\noindent
by putting $\delta':=\delta/(1-\delta).$ 
The inequality (\ref{ineq.1}) obviously implies (\ref{ineq.2}) when $\delta'=\varepsilon'$ 
as $\delta'A_{X}(F)>0$ from the log terminality assumption of $X$. 
\end{proof}

%%%%%%%%%%%%%%%%%%%%%%%%%%%%%%%%%%%%%%%%%%%%

\section{Via ``Basis type" anticanonical divisors} 

\subsection{K-stability and basis type divisors}

Our main theorem in this section is as follows. 

\begin{theorem}[K-stability criteria via basis type divisors]\label{paraph}
Let $X$ be a $\mathbb{Q}$-Fano variety.  
Let $\delta(X)$ be the value in Definition \ref{delta_dfn}.
If 
$\delta(X)>1$ $($resp., $\ge 1)$
then $(X,-K_{X})$ is uniformly K-stable $($resp., K-semistable$)$. 

\end{theorem}

\noindent 
We call the above conditions \textit{asymptotic klt} 
(resp., \textit{asymptotic log canonicity} of the (set of) log Calabi-Yau pairs $\{(X,D)\}$. 

\begin{proof}[Proof of Theorem \ref{paraph}]
Take any prime divisor $F$ over $X$. 
First we explain the key observation as the following lemma, which describes 
the minimum of 
log discrepancies of the log pairs associated to $k$-basis type divisors. 

\begin{lemma}[Log discrepancy formula]\label{ld.of.bd}
For $c>0$, we have 
$$\min_{D: k\text{{\rm -basis type}}}A_{(X,cD)}(F)=
A_{X}(F)-c\dfrac{\sum_{1\le a}h^{0}(-kK_{X}-aF)}
{k\cdot h^{0}(-kK_{X})}.$$

\noindent
In particular, the left hand side $($the minimum$)$ exists. 
\end{lemma}

\begin{proof}[Proof of Lemma \ref{ld.of.bd}]

Take any basis $s_{1}, \dots, s_{h^{0}(-kK_{X})}$ of $H^{0}(-kK_{X})$ and associate 
the $k$-basis type divisor $D$. 
Changing the order if necessary, we can and do assume that 
there is a decreasing sequence 
$h^{0}(-kK_{X})\ge i_{1}\ge \cdots\ge i_{k}\ge \cdots \ge 0$ such that 
the valuations along $F$ are $v_{F}(s_{j})=a$ if $i_{a}\ge j>i_{a+1}$. 
Then 
\[
v_{F}(D)=\dfrac{\sum_{0\le a}a(i_{a}-i_{a+1})}
{k\cdot h^{0}(-kK_{X})}
=\dfrac{\sum_{1\le a}i_{a}}
{k\cdot h^{0}(-kK_{X})}
\]
for the corresponding $k$-basis type anticanonical $\mathbb{Q}$-divisor $D$. 
Here, $v_{F}$ denotes the valuation along $F$. 
Linear independence of $\{s_{i}\}_{i}$ implies 
$i_{a}\le h^{0}(-kK_{X}-aF)$ whose equality holds for when 
the basis is compatible with the filtration $\{H^{0}(-kK_{X}-aF)\}_{a\geq 0}$. Hence, 
\[
\sup_{D: k\text{-basis type}}v_{F}(D)=\dfrac{\sum_{1\le a}h^{0}(-kK_{X}-aF)}
{k\cdot h^{0}(-kK_{X})}.
\]
\end{proof}

The definition of $\delta_{k}(X)$ implies that 
\begin{equation}\label{ineq.3}
A_{X}(F)-\delta_{k}(X)\dfrac{\sum_{1\le a}h^{0}(-kK_{X}-aF)}
{k\cdot h^{0}(-kK_{X})}\ge 0 
\end{equation}
by Lemma \ref{ld.of.bd}. 
Note that $$\lim_{k\to \infty}\dfrac{\sum_{1\le a}h^{0}(-kK_{X}-aF)}
{k\cdot h^{0}(-kK_{X})}=\frac{\int_{0}^{\tau(F)} \vol(-K_{X}-xF)dx}{(-K_{X})^{\cdot n}}.$$
Thus the assertion follows from Theorem \ref{unif_thm}.
\end{proof}

\begin{remark}
It is \textit{not} true that $(X,D)$ is always log canonical for any 
basis type anticanoncal $\mathbb{Q}$-divisor on any K-semistable Fano variety $X$. 
Indeed, we 
%Fujita 
checked by Macaulay$2$ that the Ono-Sano-Yotsutani's example \cite{OSY12} does not satisfy the condition. 

\end{remark}

%%%%%%%%%%%%%%%%%%%%%%%%%%%%%%%%%%%%%%%%%%%%%%%%

\subsection{Relation with the Berman-Gibbs stability}\label{BG_section}

In this subsection, we discuss relations of our approach to the Berman-Gibbs 
stability introduced by Berman 
\cite{Ber13}. It is also defined in terms of (pluri)anticanonical divisors 
but that of \textit{large self product} of the given Fano variety. 
First, let us recall the notion. 

\begin{definition}[{see \cite{Ber13} and \cite{fjtBG}}]\label{Berman.Gibbs}
Let $X$ be a $\mathbb{Q}$-Fano variety. Consider any $k\in\mathbb{Z}_{>0}$ with 
$-kK_X$ globally generated Cartier. We set 
\begin{itemize}
\item
$N_k:=h^0(-kK_X)$, 
\item
$\phi_k:=\phi_{|-kK_X|}\colon X\to \mathbb{P}^{N_k-1}$, 
\item
$\Phi_k:=\phi_k\times\cdots\times\phi_k\colon X^{N_k}\to(\mathbb{P}^{N_k-1})^{N_k}$,
\item
$\mathcal{D}_k
:=\Phi^*_k\Det_{N_k}$, where $\Det_{N_k}\subset(\mathbb{P}^{N_k-1})^{N_k}$ is  
the determinantal divisor, 
\item
\[
\gamma_k(X):=\lct_{\Delta_X}\left(X^{N_k}; \frac{1}{k}\mathcal{D}_k\right),
\]
where $\Delta_X\subset X^{N_k}$ is the diagonal.
\end{itemize}
Moreover, we define 
\[
\gamma(X):=\liminf_{
\substack{k\to\infty\\
-kK_X:\text{{ Cartier}}
}}\gamma_k(X).
\]
The $X$ is said to be \emph{Berman-Gibbs stable} (resp.\ \emph{Berman-Gibbs 
semistable}) if $\gamma(X)>1$ (resp.\ $\gamma(X)\geq 1$) holds. 
\end{definition}

The main purpose of this section is to 
give another simpler proof of \cite[Theorem 1.4]{fjtBG} from our perspective.

\begin{theorem}\label{BGthm}
Let $X$ be a $\mathbb{Q}$-Fano variety. 
For any $k\in\mathbb{Z}_{>0}$ with $-kK_X$ Cartier, we have the inequality 
$\delta_k(X)\geq\gamma_k(X)$. In particular, we get the inequality 
$\delta(X)\geq \gamma(X)$. 
\end{theorem}

Together with Theorem \ref{paraph}, this immediately implies the following: 

\begin{corollary}[{cf.\ \cite[\S 7]{Ber13} and \cite[Theorem 1.4]{fjtBG}}]\label{BGcor}
If $X$ is Berman-Gibbs stable 
$($resp.\ Berman-Gibbs semistable$)$, then $(X,-K_X)$ is uniformly K-stable 
$($resp.\ K-semistable$)$. 
\end{corollary}

\begin{proof}[Proof of Theorem \ref{BGthm}]
Take any prime divisor $F$ over $X$ and take any log resolution $\sigma\colon Y\to X$ 
with $F$ smooth divisor on $Y$. 
For any $k\in\mathbb{Z}_{>0}$ with $-kK_X$ Cartier, set $\psi_k:=\phi_k\circ\sigma
\colon Y\to\mathbb{P}^{N_k-1}$. Then $\psi_k=\phi_{|\sigma^*(-kK_X)|}$, and 
$\Psi_k:=(\phi_k\circ\sigma)^{N_k}=\Phi_k\circ\sigma^{N_k}\colon
Y^{N_k}\to(\mathbb{P}^{N_k-1})^{N_k}$ satisfies that $\Psi_k^*\Det_{N_k}=(\sigma^{N_k}
)^*\mathcal{D}_k$. Thus the multiplicity of 
$(\sigma^{N_k})^*\mathcal{D}_k$ along $F^{N_k}=F\times\cdots\times F\subset Y^{N_k}$
is bigger than or equal to 
\[
\sum_{j=1}^\infty h^0(-kK_X-jF).
\]
On the other hand, we have 
\[
(\sigma^{N_k})^*K_{X^{N_k}}=K_{Y^{N_k}}-\sum_{
\substack{F_l\subset Y\\
\sigma\text{{-exceptional}}
}}\sum_{i=1}^{N_k}(A_X(F_l)-1)p_i^*F_l,
\]
where $p_i\colon X^{N_k}\to X$ is the $i$-th projection morphism. 
Let $\mathcal{Z}\to Y^{N_k}$ be the blowup along $F^{N_k}$ and let 
$G\subset\mathcal{Z}$ be the exceptional divisor. Then we have 
\begin{eqnarray*}
0\leq A_{\left(X^{N_k}, \frac{\gamma_k(X)}{k}\mathcal{D}_k\right)}(G)
\leq
N_k\cdot A_X(F)-\frac{\gamma_k(X)}{k}\sum_{j=1}^{\infty}h^0(-kK_X-jF).
\end{eqnarray*}
This implies that 
\[
A_X(F)\geq \frac{\gamma_k(X)}{k\cdot N_k}\sum_{j=1}^{\infty}
h^0(-kK_X-jF).
\]
Thus $A_{(X, \gamma_k(X)D)}(F)\geq 0$ for any $k$-basis type divisor $D$ and for 
any prime divisor $F$ over $X$ by Lemma \ref{ld.of.bd}. Hence we get 
the inequality $\delta_k(X)\geq \gamma_k(X)$. 
\end{proof}

\begin{remark}
From Theorem \ref{BGthm}, it also follows that if the Berman-Gibbs stability is 
equivalent to the K-stability (cf., \cite[\S 7]{Ber13}), our Conjecture \ref{Main.conj} 
is also true. 

\end{remark}

%%%%%%%%%%%%%%%%%%%%%%%%%%%%%%

\section{Relation with the alpha invariant}

Let us recall the basic of the alpha invariants \cite{Tia87} which was first introduced in 
terms of the K\"ahler potentials. Later it was proved to be the same as the following algebraic version \textit{global log canonical 
threshold} which we use. 

\begin{theorem}[cf., \cite{Dem08}]\label{Dem}
For an arbitrary Fano manifold $X$, 
$$\alpha(X)=\sup\{\alpha>0\mid (X,\alpha D) \mbox{ is log canonical for any effective } D
\sim_\mathbb{Q} -K_{X}\}.$$
\end{theorem}

In this paper, we treat the right hand side as the definition of alpha invariant as 
we do not think there would be some confusion. 
In particular, the definition naturally extends to $\mathbb{Q}$-Fano varieties. 

The purpose of this section is to give an algebraic new proof of the following theorem 
from our perspective. 

\begin{theorem}[cf., \cite{Tia87,OS12,Der14,BHJ15}]\label{alpha.K}
For a $\mathbb{Q}$-Fano variety $X$ of dimension $n$, 
if $\alpha(X)> (\text{resp., } \ge) n/(n+1)$, then $(X,-K_X)$ 
is uniformly K-stable $($resp., K-semistable$)$. 
\end{theorem}

Given the equivalence \ref{Dem} above, the above result can also be seen as a 
K-stability criterion in terms of anticanonical $\mathbb{Q}$-divisors. 
The statement is an algebraic counterpart of Tian's original statement that 
``For a Fano manifold $X$ with $\alpha(X)>n/(n+1)$, there exists a K\"ahler-Einstein metric'' \cite{Tia87}. First algebraic proof of the K-stability (Theorem \ref{alpha.K}) was obtained by the first author and 
Sano \cite{OS12} and then was later 
refined to the uniform K-stability by Dervan \cite{Der14} and Boucksom-Hisamoto-Jonsson 
\cite{BHJ15}. Both analyze the 
test configurations, thus the arguments are based on 
$(n+1)$-dimensional birational geometry. Our proof here 
is very different, in particular, is based on $n$-dimensional birational geometry 
and make use of the Okounkov body. 

We prepare the following lemma in order to prove Theorem \ref{alpha.K}.

\begin{lemma}\label{alph_lem}
For any $\mathbb{Q}$-Fano variety $X$ and for any prime divisor $F$ over $X$, 
we have $\alpha(X)\cdot\tau(F)\leq A_X(F).$
\end{lemma}

\begin{proof}
Assume that $k\in\mathbb{Z}_{>0}$ and $\tau\in\mathbb{R}_{\geq 0}$ satisfies that 
$H^0(X, -kK_X-k\tau F)\neq 0$. Then we can find an effective 
$\mathbb{Q}$-divisor $D$ 
with $D\sim_\mathbb{Q} -K_X$ such that $A_{(X, (A_X(F)/\tau)D)}(F)\leq 0$. Thus we have 
$A_X(F)/\tau\geq \alpha(X)$. 
\end{proof}

\begin{proof}[Proof of Theorem \ref{alpha.K}]
Assume that $\alpha(X)>n/(n+1)$ (resp.\ $\geq n/(n+1)$). 
We can take $\delta\in(0,1)$ (resp.\ $\delta\in[0,1)$) such that 
$\delta\leq(n+1)(\alpha(X)-(n/(n+1)))$. 
Pick any dreamy prime divisor $F$ over $X$ in the sense of \cite{Fjt16}, that is, 
the graded algebra
\[
\bigoplus_{k,j\in\mathbb{Z}_{\geq 0}}H^0(-kK_X-jF)
\]
is finitely generated.

\begin{claim}\label{alph_claim}
There exists a normal projective variety $X'$ with $-K_{X'}$ $\mathbb{Q}$-Cartier 
and a prime divisor $F'$ on $X'$ which is $\mathbb{Q}$-Cartier such that: 
\begin{itemize}
\item
for any $k\in\mathbb{Z}_{>0}$ and $x\in\mathbb{R}_{>0}$ with $-kK_X$ Cartier, we 
have the equality
\[
H^0(-kK_X-xF)=H^0(X', k(-K_{X'}+(A_X(F)-1)F')-xF'),
\]
\item
for any $0<\varepsilon\ll 1$, $(-K_{X'}+(A_X(F)-1)F')-\varepsilon F'$ is ample. 
\end{itemize}
\end{claim}

\begin{proof}[Proof of Claim \ref{alph_claim}]
Take any projective birational morphism $\sigma\colon Y\to X$ with $Y$ smooth and 
$F\subset Y$.
By \cite[Theorem 4.2]{KKL} and \cite[Claim 6.3]{Fjt16}, 
we can find a birational contraction map 
$\varphi\colon Y\dashrightarrow X'$ such that the strict transform $F'\subset X'$ of 
$F$ satisfies that the map $\varphi$ is the ample model of 
$(-K_{X'}+(A_X(F)-1)F')-\varepsilon F'$ for any $0<\varepsilon\ll 1$. 
(Indeed, the strict transform of $\sigma^*(-K_X)$ on $X'$ is equal to 
$-K_{X'}+(A_X(F)-1)F'$.) The divisor $\sigma^*(-K_X)$ is of course 
$\varphi$-nonnegative, and the divisor $\sigma^*(-K_X)-\varepsilon F$ is 
$\varphi$-nonpositive for any $0<\varepsilon\ll 1$. Thus 
the divisor $\sigma^*(-K_X)-x F$ is 
$\varphi$-nonpositive for any $x\in\mathbb{R}_{>0}$. 
Thus the assertion follows from \cite[Remark 2.4 (i)]{KKL}. 
\end{proof}

Set $H':=-K_{X'}+(A_X(F)-1)F'$. Fix any rational number $0<\varepsilon\ll 1$. 
Take any complete flag in $X'$ (in the sense of \cite{LM})
\[
X'\supset Z_1\supset Z_2\supset\cdots\supset Z_n=\{\text{point}\}
\]
with $Z_1=F'$. Let us consider the Okounkov body 
$\Delta_\varepsilon:=\Delta_{Z_\bullet}(H'-\varepsilon F')\subset \mathbb{R}_{\geq 0}^n$ 
of $H'-\varepsilon F'$ with respects to the flag $Z_\bullet$ (see \cite{LM} for the 
definition of the Okounkov bodies). 
Since $H'-\varepsilon F'$ is ample, by \cite[Corollary 4.25]{LM} 
and Claim \ref{alph_claim}, we have 
\[
\vol(\Delta_\varepsilon|_{\nu_1\geq x-\varepsilon})=\frac{1}{n!}\vol_{X'}(H'-xF')
=\frac{1}{n!}\vol(-K_X-xF)
\]
for any $x\geq \varepsilon$, where 
\[
\Delta_\varepsilon|_{\nu_1\geq x-\varepsilon}:=\{(\nu_1,\dots,\nu_n)\in
\Delta_\varepsilon\,|\, \nu_1\geq x-\varepsilon\}.
\]
For any $x\geq \varepsilon$, 
let $Q(x)$ be the restricted volume of 
\[
\{(\nu_1,\dots,\nu_n)\in
\Delta_\varepsilon\,|\, \nu_1=x-\varepsilon\}.
\]
Then we have 
\begin{eqnarray*}
&&\int_\varepsilon^{\tau(F)}\vol(-K_X-xF)dx
=\int_\varepsilon^{\tau(F)}\int_x^{\tau(F)}n!\cdot Q(y)dydx\\
&=&\int_\varepsilon^{\tau(F)}\int_\varepsilon^yn!\cdot Q(y)dxdy
=\int_\varepsilon^{\tau(F)}n!\cdot(y-\varepsilon)Q(y)dy.
\end{eqnarray*}
Thus we get 
\[
\frac{\int_\varepsilon^{\tau(F)}\vol(-K_X-xF)dx}{\vol(-K_X-\varepsilon F)}=
\frac{\int_\varepsilon^{\tau(F)}(y-\varepsilon)Q(y)dy}{\int_\varepsilon^{\tau(F)}
Q(y)dy}.
\]
The right-hand side is nothing but the first coordinate (say, $b_1$) of the barycenter of 
$\Delta_\varepsilon$. The value $b_1$ satisfies that 
$b_1\leq (\tau(F)-\varepsilon)\cdot n/(n+1)$ (see for example \cite{hammer}). 
This implies that 
\[
\int_0^{\tau(F)}\vol(-K_X-xF)dx\leq\frac{n}{n+1}\tau(F)\vol(-K_X).
\]
Thus the assertion follows from Lemma \ref{alph_lem} and Theorem \ref{unif_thm} 
(or, by \cite[Theorem 1.3]{Fjt16}). 
\end{proof}

We end our notes by observing the following lower bound of alpha invariant which is 
somewhat easier. 

\begin{theorem}\label{alpha.K2}
If a $\mathbb{Q}$-Fano variety $X$ of dimension $n$ is 
K-semistable, then $\alpha(X)\ge 1/(n+1)$ holds. 
\end{theorem}

\begin{proof}[Proof of Theorem \ref{alpha.K2}]
Take any $k\in\mathbb{Z}_{>0}$ with $-kK_X$ Cartier and take any $D\in|-kK_X|$. 
Then, by \cite[Theorem 4.10]{fjt2}, we have 
\[
\lct(X; D)\cdot(-K_X)^{\cdot n}\geq\int_0^{\infty}\vol(-K_X-xD)dx.
\]
Since $\vol(-K_X-xD)=(1-kx)^n(-K_X)^{\cdot n}$ (if $0\leq x\leq 1/k$), 
we have $\lct(X; D)\geq 1/(k(n+1))$. 
\end{proof}

\begin{remark}\label{liu_rmk}
Yuchen Liu independently obtained Theorem \ref{alpha.K2} (\cite{liu}). 
\end{remark}

\bigskip
{\small
Research Institute for Mathematical Sciences, Kyoto University, Kyoto 606-8502. JAPAN 
\begin{flushright}
\texttt{fujita@kurims.kyoto-u.ac.jp}
\end{flushright}
\vspace{0.5cm}

Department of Mathematics, Kyoto University, Kyoto 606-8502. JAPAN 
\begin{flushright}
\texttt{yodaka@math.kyoto-u.ac.jp}
\end{flushright}
}


\begin{thebibliography}{99}



\bibitem[Ber13]{Ber13}
R.\ Berman, \emph{K\"ahler-Einstein metrics, canonical random point processes 
and birational geometry}, 
arXiv:1307.3634.

\bibitem[Ber16]{Ber12}
R.\ Berman, \emph{K-polystability of Q-Fano varieties admitting 
K\"ahler-Einstein metrics}, 
Invent.\ Math.\ \textbf{203} (2016), no.\ 3, 973--1025.

\bibitem[BHJ15]{BHJ15}
S.~Boucksom, T.~Hisamoto and M.~Jonsson, 
\emph{Uniform K-stability, Duistermaat-Heckman measures and singularities of pairs}, 
arXiv:1504.06568. 

\bibitem[CDS15a]{CDS1}
X.\ Chen, S.\ Donaldson and S.\ Sun, \emph{K\"ahler-Einstein metrics on 
Fano manifolds, 
I: approximation of metrics with cone singularities}, 
J.\ Amer.\ Math.\ Soc.\ \textbf{28} (2015), no.\ 1, 183--197. 

\bibitem[CDS15b]{CDS2}
X.\ Chen, S.\ Donaldson and S.\ Sun, \emph{K\"ahler-Einstein metrics on 
Fano manifolds, 
II: limits with cone angle less than $2\pi$}, 
J.\ Amer.\ Math.\ Soc.\ \textbf{28} (2015), no.\ 1, 199--234. 

\bibitem[CDS15c]{CDS3}
X.\ Chen, S.\ Donaldson and S.\ Sun, \emph{K\"ahler-Einstein metrics on 
Fano manifolds, 
III: limits as cone angle approaches $2\pi$ and completion of the main proof}, 
J.\ Amer.\ Math.\ Soc.\ \textbf{28} (2015), no.\ 1, 235--278. 

\bibitem[CT08]{CT}
X.\ Chen and G.\ Tian, \emph{Geometry of K\"ahler metrics and foliations by 
holomorphic discs}, Publ.\ Math.\ Inst.\ Hautes \'Etudes Sci.\ \textbf{107} 
(2008), 1--107. 

\bibitem[Dem08]{Dem08}
J.-P.\ Demailly, 
Appendix to I.\ Cheltsov and C.\ Shramov's article. 
``Log canonical thresholds of smooth Fano threefolds” : \emph{On Tian's invariant and log canonical thresholds}, 
Russian Math.\ Surveys \textbf{63} (2008), no.\ 5, 945--950. 

\bibitem[Der14]{Der14}
R.~Dervan, 
\emph{Uniform stability of twisted constant scalar curvature K\"ahler metrics}, 
arXiv:1412.0648, to appear in Int.\ Math.\ Res.\ Not.\ IMRN.

\bibitem[Don02]{Don02}
S.~Donaldson, 
\emph{Scalar curvature and stability of toric varieties}, 
J.\ Differential Geom.\ \textbf{62} (2002), no.\ 2, 289--349. 

\bibitem[Don05]{don05}
S.\ Donaldson, \emph{Lower bounds on the Calabi functional}, 
J.\ Differential Geom.\ \textbf{70} (2005), no.\ 3, 453--472. 

\bibitem[DS14]{DS14}
S.~Donaldson and S.~Sun, 
\emph{Gromov-Hausdorff limits of Kahler manifolds and algebraic geometry}, 
Acta Math.\ \textbf{213} (2014), no.\ 1, 63--106. 

\bibitem[DT92]{DT92}
W.~Ding and G.~Tian, 
\emph{K\"ahler-Einstein metrics and the generalized Futaki invariants}, 
Invent.\ Math.\ \textbf{110}, (1992), no.\ 2, 315--335. 


\bibitem[Fjt15a]{Fjt15}
K.~Fujita, 
\emph{On K-stability and the volume functions of $\mathbb{Q}$-Fano varieties}, 
arXiv:1508.04052; accepted by Proc.\ London Math.\ Soc. 

\bibitem[Fjt15b]{fjt2}
K.\ Fujita, \emph{Optimal bounds for the volumes of K\"ahler-Einstein Fano manifolds}, 
arXiv:1508.04578; accepted by Amer.\ J.\ Math.

\bibitem[Fjt16a]{fjtBG}
K.\ Fujita, \emph{On Berman-Gibbs stability and K-stability of $\mathbb{Q}$-Fano 
varieties}, Compos.\ Math.\ \textbf{152} (2016), no.\ 2, 288--298.

\bibitem[Fjt16b]{Fjt16}
K.~Fujita, 
\emph{A valuative criterion for uniform K-stability of $\mathbb{Q}$-Fano varieties}, 
arXiv:1602.00901; accepted by J.\ Reine Angew.\ Math. 

\bibitem[FS90]{FS90}
A.~Fujiki and G.~Schumacher, 
\emph{The moduli space of extremal compact K\"ahler manifolds and generalized 
Weil-Petersson metrics}, 
Publ.\ Res.\ Inst.\ Math.\ Sci.\ \textbf{26} (1990), no.\ 1, 101--183. 

\bibitem[Ham51]{hammer}
P.\ C.\ Hammer, \emph{The centroid of a convex body}, Proc.\ Amer.\ Math.\ Soc.\ 
\textbf{2}, (1951), 522--525.

\bibitem[KKL12]{KKL}
A.-S.\ Kaloghiros, A.\ K\"uronya and V.\ Lazi\'c, \emph{Finite generation and 
geography of models}, arXiv:1202.1164; to appear in Advanced Studies in Pure Mathematics, Mathematical Society of Japan, Tokyo.

\bibitem[KM98]{KM98}
J.\ Koll{\'a}r and S.\ Mori, \emph{Birational geometry of algebraic varieties},
With the collaboration of C.\ H.\ Clemens and A.\ Corti. 
Cambridge Tracts in Math., \textbf{134},
Cambridge University Press, Cambridge, 1998.

\bibitem[Li16]{Li16}
C.~Li, 
\emph{K-semistability is equivariant volume minimization}, 
arXiv:1512.07205v2. 

\bibitem[Liu]{liu}
Y.\ Liu, \emph{private communication}.

\bibitem[LM09]{LM}
R.\ Lazarsfeld and M.\ Musta\c{t}\u{a}, \emph{Convex bodies associated 
to linear series}, 
Ann.\ Sci.\ \'Ec.\ Norm.\ Sup\'er.\ \textbf{42} (2009), no.\ 5, 783--835. 

\bibitem[LWX14]{LWX14} 
C.~Li, X.~Wang and C.~Xu, 
\emph{Degeneration of Fano K\"ahler-Einstein varieties}, 
arXiv:1411.0761

\bibitem[Mab08]{mab1}
T.\ Mabuchi, \emph{K-stability of constant scalar curvature}, arXiv:0812.4903. 

\bibitem[Mab09]{mab2}
T.\ Mabuchi, \emph{A stronger concept of K-stability}, arXiv:0910.4617. 

\bibitem[Od10]{Od10}
Y.~Odaka, 
\emph{On the GIT stability of polarised varieties - a survey -}, 
Proceeding of Kinosaki algebraic geometry symposium 2010. 

\bibitem[Od13]{Od13b}
Y.~Odaka, 
\emph{On the moduli of K\"ahler-Einstien Fano manifolds}, 
Proceeding of Kinosaki algebraic geometry symposium 2013, 
available at arXiv:1211.4833. 

\bibitem[Od15]{Od15}
Y.~Odaka, 
\emph{Compact moduli spaces of K\"ahler-Einstein Fano varieties}, 
Publ.\ Res.\ Inst.\ Math.\ Sci.\ \textbf{51} (2015), no.\ 3, 549--565. 

\bibitem[OS12]{OS12}
Y.~Odaka and Y.~Sano, 
\emph{Alpha invariants and K-stability of $\mathbb{Q}$-Fano varieties}, 
Adv.\ Math.\ \textbf{229} (2012), no.\ 5, 2818--2834. 

\bibitem[OSS16]{OSS16}
Y.~Odaka, C.~Spotti and S.~Sun, 
\emph{Compact moduli spaces of Del Pezzo surfaces and K\"ahler-Einstein metrics}, 
J.\ Differential Geom.\ \textbf{102} (2016), no.\ 1, 127--172. 

\bibitem[OSY12]{OSY12}
H.~Ono, Y.~Sano and N.~Yotsutani, \emph{An example of an asymptotically Chow unstable manifold with constant scalar curvature}, 
Ann.\ Inst.\ Fourier (Grenoble) \textbf{62} (2012), no.\ 4, 1265--1287.

\bibitem[SSY14]{SSY14}
C.~Spotti, S.~Sun and C.~Yao, 
\emph{Existence and deformations of K\"ahler-Einstein metrics on smoothable 
$\mathbb{Q}$-Fano varieties}, 
arXiv:1411.1725. 

\bibitem[Sto09]{stoppa}
J.\ Stoppa, \emph{K-stability of constant scalar curvature K\"ahler manifolds}, 
Adv.\ Math.\ \textbf{221} (2009), no.\ 4, 1397--1408. 

\bibitem[Tia87]{Tia87}
G.~Tian, 
\emph{On K\"ahler-Einstein metrics on certain K\"ahler manifolds with 
$C_1(M)>0$}, Invent.\ Math.\ \textbf{89} (1987), no.\ 2, 225--246.


\bibitem[Tia97]{tian1}
G.\ Tian, \emph{K\"ahler-Einstein metrics with positive scalar curvature}, 
Invent.\ Math.\ \textbf{130} (1997), no.\ 1, 1--37.


\bibitem[Tia15]{tian2}
G.\ Tian, \emph{K-stability and K\"ahler-Einstein metrics}, 
Comm. Pure Appl.\ Math.\ \textbf{68} (2015), no.\ 7, 1085--1156; 
Corrigendum, 
\textbf{68} (2015), no.\ 11, 2082--2083. 



\end{thebibliography}
\end{document}